\documentclass[11pt, twoside]{article}
\usepackage{latexsym}
\usepackage{amsmath}
\usepackage{amssymb}
\usepackage[all]{xy}
\usepackage{amsfonts}
\usepackage{verbatim}
\usepackage{amsthm}
\usepackage{mathrsfs}
\usepackage{epsfig}
\usepackage{xy}
\usepackage{array}
\usepackage{stmaryrd}
\usepackage{graphicx,color}
\usepackage{xcolor}
\usepackage[colorlinks=true,linkcolor=blue,citecolor=blue]{hyperref}
\usepackage{tikz}
\usetikzlibrary{arrows,calc}
\usepackage{etex}
\usepackage{mathdots}
\usepackage{float}
\usepackage{graphics}
\usepackage{pdflscape}

\usepackage{anysize,hyperref}
\input xypic
\xyoption{all}
\usepackage{enumitem}
\usepackage[perpage,symbol]{footmisc}
\usepackage{setspace}
\numberwithin{equation}{section}

\newcommand{\add}{\operatorname{add}}
\newcommand{\thick}{\operatorname{thick}}
\newcommand{\pr}{\operatorname{pr}}
\newcommand{\proj}{\operatorname{proj}}
\newcommand{\Hom}{\operatorname{Hom}}
\newcommand{\End}{\operatorname{End}}

\newcommand{\Cone}{\operatorname{Cone}}

\newcommand{\pd}{\operatorname{pd}}
\newcommand{\PP}{\mathbb P}
\usepackage{bm}
\begin{document}
\baselineskip=15pt
\title{\Large{\bf A negative answer to a question on tilting objects and\\[2mm] two-term complexes\footnotetext{Jing He is supported by the National Natural Science Foundation of China (Grant No. 12401045). Panyue Zhou is supported by the National Natural Science Foundation of China (Grant No. 12371034).
}}}
\medskip
\author{Jing He and Panyue Zhou}

\date{}

\maketitle
\def\blue{\color{blue}}
\def\red{\color{red}}

\theoremstyle{plain}
\newtheorem{theorem}{Theorem}[section]
\newtheorem{proposition}[theorem]{Proposition}
\newtheorem{lemma}[theorem]{Lemma}
\newtheorem{corollary}[theorem]{Corollary}
\theoremstyle{definition}
\newtheorem{definition}[theorem]{Definition}
\newtheorem{remark}[theorem]{Remark}

\baselineskip=17pt
\parindent=0.5cm
\vspace{-6mm}

\begin{abstract}
\baselineskip=16pt
Let $M$ be a silting object in an idempotent complete algebraic triangulated category $\mathcal T$.
Put $B=\End_{\mathcal T}(M)$, and let
$\PP_M\colon \pr(M)\to K^{[-1,0]}(\proj B)$ be the presentation functor associated with $M$.
It was recently asked whether $\PP_M(T)$ must be tilting whenever $T\in\pr(M)$ is a tilting object.
We answer this question in the negative by giving an explicit finite-dimensional example.
Namely, for
\[
\Lambda=k(1\xrightarrow{\alpha}2\xrightarrow{\beta}3\xrightarrow{\gamma}4)/(\alpha\beta\gamma)
\]
we construct a silting object $M\in K^b(\proj\Lambda)$ and a tilting object $T=\Sigma\Lambda\in\pr(M)$ for which
\[
\dim_k\Hom_{K^b(\proj B)}\bigl(\PP_M(T),\Sigma^{-1}\PP_M(T)\bigr)=1.
\]
Thus $\PP_M(T)$ is a two-term silting complex but not a tilting complex.
\\[0.2cm]
\textbf{Keywords:} silting object, tilting object, two-term silting complex, triangulated category, homotopy category\\[0.1cm]
\textbf{Mathematics Subject Classification 2020:}  16E35; 18G80; 16G20
\end{abstract}

\pagestyle{myheadings}
\markboth{\rightline {\scriptsize   Jing He and Panyue Zhou\hspace{6mm}}}
         {\leftline{\scriptsize A negative answer to tilting question}}

\section{Introduction}

Silting theory can be viewed as a flexible extension of classical tilting theory in triangulated and derived categories. A silting object is required to have no positive self-extensions and to generate the ambient triangulated category, while a tilting object has no self-extensions in any nonzero degree. Thus, a construction that preserves the vanishing of positive extensions may still introduce negative self-extensions, making the distinction between silting and tilting genuinely significant. The example presented in this paper provides a small and explicit illustration of this phenomenon.

Let $\mathcal T$ be an idempotent complete algebraic triangulated category with shift functor $\Sigma$, and let $M$ be a rigid object. Denote by $\pr(M)$ the full subcategory consisting of objects $X$ admitting a triangle
$$
M^{-1}\longrightarrow M^0\longrightarrow X\longrightarrow \Sigma M^{-1}
\quad \mbox{with}~M^{-1},M^0\in\add(M).
$$
Yang \cite{Yang2025} constructed a presentation functor
$$
\PP_M\colon \pr(M)\longrightarrow K^{[-1,0]}(\proj A),
~~\mbox{where}~A=\End_{\mathcal T}(M),
$$
which associates to $X$ the two-term complex arising from such a presentation. When $M$ is silting, this functor induces a bijection between the isomorphism classes of two-term silting objects and two-term silting complexes \cite[Theorem~3.14(c)]{Yang2025}. Yang then asked whether, for a tilting object $T\in\pr(M)$, the complex $\PP_M(T)$ must also be tilting \cite[Question~3.15]{Yang2025}. It remains to determine whether a negative self extension can occur after applying $\PP_M$.

For finite-dimensional algebras, this question is closely connected with the silting theorem of Buan and Zhou \cite{BuanZhou2016}. Starting with a two-term silting complex, their construction yields a two-term silting complex over its endomorphism algebra. Xie, Yang, and Zhang \cite[Section~5]{XieYangZhang2023} later studied when this new complex is tilting and established several sufficient conditions, including the hereditary and symmetric cases. They also gave a description of its negative self extension space in terms of morphisms satisfying two annihilation conditions. In the example considered below, this space is one-dimensional.

Our main result is the following.

\begin{theorem}\label{thm:main}
Let $k$ be a field and let
\[
\Lambda=
 k(1\xrightarrow{\alpha}2\xrightarrow{\beta}3\xrightarrow{\gamma}4)
 /(\alpha\beta\gamma),
\]
where paths are multiplied from left to right.
For $P_i=e_i\Lambda$, consider in $K^b(\proj\Lambda)$ the complexes
\[
X=(P_2\xrightarrow{\alpha}P_1),~
Y=(P_4\xrightarrow{\gamma}P_3),~
Z=P_3,~
W=\Sigma P_2,
\]
with $X$ and $Y$ concentrated in degrees $-1$ and $0$, and put
\[
M=X\oplus Y\oplus Z\oplus W.
\]
Then the following statements hold.
\begin{enumerate}
\item[\rm (1)]$M$ is a silting object of $K^b(\proj\Lambda)$.
\item[\rm (2)]$T=\Sigma\Lambda$ is a tilting object of $K^b(\proj\Lambda)$ and belongs to $\pr(M)$.
\item[\rm (3)] If $B=\End_{K^b(\proj\Lambda)}(M)$, then
\[
\dim_k\Hom_{K^b(\proj B)}
\bigl(\PP_M(T),\Sigma^{-1}\PP_M(T)\bigr)=1.
\]
Consequently, $\PP_M(T)$ is not tilting.
\end{enumerate}
In particular, the answer to {\rm\cite[Question~3.15]{Yang2025}} is negative in general.
\end{theorem}

The obstruction is represented by the path
$$
P_4\xrightarrow{\beta\gamma}P_2.
$$
The relations
$$
\beta\gamma\neq0,~\alpha\beta\neq0,~
\alpha\beta\gamma=0
$$
enter the argument in different ways. The first ensures that the obstruction is nonzero, the second shows that it is not null-homotopic, and the third guarantees that it defines a chain map. Together, these relations account for the negative self extension and show why the standard sufficient criteria are not applicable in this example.

The paper is organized as follows. In Section~2, we fix the notation and recall the presentation functor used throughout. In Section~3, we construct the silting object and a presentation of $\Sigma\Lambda$. In Section~4, we determine the negative self extension space and prove Theorem~\ref{thm:main}. In Section~5, we compare the example with the projective-dimension criterion in \cite{XieYangZhang2023}.

\section{Preliminaries}

Throughout this paper, $k$ denotes a field, and all categories are assumed to be $k$-linear. A triangulated category is called algebraic if it is triangle equivalent to the stable category of a Frobenius exact category.

For an additive category $\mathcal A$, we denote by $K^b(\mathcal A)$ the bounded homotopy category of cochain complexes over $\mathcal A$. Let $K^{[-1,0]}(\mathcal A)$ be the full subcategory consisting of objects isomorphic to complexes concentrated in degrees $-1$ and $0$. In \cite{Yang2025}, this subcategory is denoted by $\mathcal H^{[-1,0]}(\mathcal A)$. We use the shift convention
$$
(\Sigma C)^i=C^{i+1},~ d_{\Sigma C}^i=-d_C^{i+1}.
$$
A complex concentrated in degrees $-1$ and $0$ will be called a two-term complex.

We first recall the notions of silting and tilting objects that will be used in the sequel.
Let $\mathcal T$ be a triangulated category. For an object $U\in\mathcal T$, we write $\add(U)$ for the full subcategory of direct summands of finite direct sums of copies of $U$, and $\thick(U)$ for the smallest thick subcategory of $\mathcal T$ containing $U$.

\begin{definition}\cite[Definition 2.1]{AiharaIyama2012}
An object $U\in\mathcal T$ is \emph{presilting} if
$$
\Hom_{\mathcal T}(U,\Sigma^iU)=0
~~\text{for all }~ i>0.
$$
It is \emph{silting} if it is presilting and $\thick(U)=\mathcal T$. It is \emph{tilting} if
$$
\Hom_{\mathcal T}(U,\Sigma^iU)=0
~~\text{for all } ~i\neq0
$$
and $\thick(U)=\mathcal T$.
\end{definition}

We next recall the construction associated with a rigid object that will be used in the proof.
Let $M$ be a rigid object in an idempotent complete algebraic triangulated category $\mathcal T$, that is,
$$
\Hom_{\mathcal T}(M,\Sigma M)=0.
$$
Following \cite{Yang2025}, we consider the full subcategory
$$
\pr(M)=\left\{X\in\mathcal T\ \middle|\
\begin{array}{c}
\text{there exists a triangle }M^{-1}\to M^0\to X\to\Sigma M^{-1}\\
\text{with }M^{-1},M^0\in\add(M)
\end{array}
\right\}.
$$
Set $A=\End_{\mathcal T}(M)$. The restricted Yoneda functor
$$
F_M=\Hom_{\mathcal T}(M,-)\colon\add(M)\longrightarrow\proj A
$$
is an equivalence. For a presentation
$$
M^{-1}\xrightarrow{~a~}M^0\longrightarrow X\longrightarrow\Sigma M^{-1},
$$
the presentation functor introduced in \cite{Yang2025} sends $X$, up to isomorphism, to the two-term complex
$$
\PP_M(X)=\bigl(F_M(M^{-1})\xrightarrow{F_M(a)}F_M(M^0)\bigr).
$$
If $M$ is silting, then $\PP_M$ induces a correspondence between two-term silting objects in $\pr(M)$ and two-term silting complexes over $A$, see \cite[Theorem~3.14]{Yang2025}.

We will use the following description of negative extensions for two-term complexes. It is the two-term version of the obstruction considered in \cite[Lemma~5.4]{XieYangZhang2023}. We include a proof for completeness.

\begin{lemma}\label{lem:obstruction}
Let $\mathcal A$ be an additive category and let
$$
S=(S^{-1}\xrightarrow{d}S^0)
$$
be a two-term complex.
Then
$$
\Hom_{K^b(\mathcal A)}(S,\Sigma^{-1}S)
\cong
\{u\in\Hom_{\mathcal A}(S^0,S^{-1})\mid ud=0=du\}.
$$
In particular, no further quotient by chain homotopies is needed on the right-hand side.
\end{lemma}

\begin{proof}
A chain map from $S$ to $\Sigma^{-1}S$ has at most one nonzero component,
namely a morphism $u\colon S^0\to S^{-1}$.
The chain map conditions are exactly $ud=0=du$.
Moreover, any homotopy between such chain maps is zero in the relevant degree, since both complexes are concentrated in degrees $-1$ and $0$. Hence the chain maps coincide with the morphisms in the homotopy category.
\end{proof}

\section{The counterexample}

{\bf 3.1 The algebra and its projectives.}
~~We now describe the algebra used in the counterexample. Its small size allows us to compute the relevant projective modules and morphisms explicitly.

Let
$$
\Lambda=
 k(1\xrightarrow{\alpha}2\xrightarrow{\beta}3\xrightarrow{\gamma}4)
 /(\alpha\beta\gamma).
$$
We compose paths from left to right. Thus $\alpha\beta$ is the path from $1$ to $3$, while $\beta\gamma$ is the path from $2$ to $4$. The residue classes of
$$
e_1,e_2,e_3,e_4,\alpha,\beta,\gamma,\alpha\beta,\beta\gamma
$$
form a $k$-basis of $\Lambda$. Hence $\dim_k\Lambda=9$, and we have
\begin{equation}\label{eq:path-relations}
\alpha\beta\neq0,~
\beta\gamma\neq0,~
\alpha\beta\gamma=0.
\end{equation}

Let $P_i=e_i\Lambda$ be the indecomposable projective right $\Lambda$-modules. For right projective modules, evaluation at $e_i$ yields a natural isomorphism
$$
\Hom_\Lambda(P_i,P_j)\cong e_j\Lambda e_i.
$$
Accordingly, a path from $i$ to $j$ determines a morphism from $P_j$ to $P_i$. We use the same notation for the path and the corresponding morphism. In particular,
$$
\alpha:P_2\to P_1,~
\beta:P_3\to P_2,~
\gamma:P_4\to P_3.
$$
The Hom spaces needed later are
\begin{equation}\label{eq:homspaces-projectives}
\begin{alignedat}{3}
\Hom(P_2,P_1)&=k\alpha, \qquad&
\Hom(P_3,P_2)&=k\beta, \qquad&
\Hom(P_4,P_3)&=k\gamma,\\
\Hom(P_3,P_1)&=k\alpha\beta, &
\Hom(P_4,P_2)&=k\beta\gamma, &
\Hom(P_4,P_1)&=0.
\end{alignedat}
\end{equation}
while $\End(P_i)=k$ for all $i$.
\vspace{2mm}

\hspace{-5mm}{\bf 3.2 A silting object.}~ We now construct the silting object that will give the desired counterexample.
The category $K^b(\proj\Lambda)$ is an idempotent complete algebraic triangulated $k$-category, via the degreewise split Frobenius structure on bounded complexes of projective $\Lambda$-modules.
In $K^b(\proj\Lambda)$, we define
$$
X=(P_2\xrightarrow{\alpha}P_1),~
Y=(P_4\xrightarrow{\gamma}P_3),~
Z=P_3,~
W=\Sigma P_2,
$$
and set
$$
M=X\oplus Y\oplus Z\oplus W.
$$

\begin{proposition}\label{prop:M-silting}
The object $M$ is silting in $K^b(\proj\Lambda)$.
\end{proposition}

\begin{proof}
For two-term complexes
$$
U=(U^{-1}\xrightarrow{~d_U~}U^0),~
V=(V^{-1}\xrightarrow{~d_V~}V^0),
$$
there exists an isomorphism
$$
\Hom_{K^b(\proj\Lambda)}(U,\Sigma V)
\cong
\frac{\Hom_\Lambda(U^{-1},V^0)}
{d_V\Hom_\Lambda(U^{-1},V^{-1})+
 \Hom_\Lambda(U^0,V^0)d_U}.
$$
A direct computation using \eqref{eq:homspaces-projectives} gives
\[
\begin{array}{c|cccc}
\Hom_{K^b(\proj\Lambda)}(-,\Sigma -) & X & Y & Z & W\\[2mm] \hline
X & k\alpha/k\alpha & 0 & 0 & 0\\[2mm]
Y & 0 & k\gamma/k\gamma & k\gamma/k\gamma & 0\\[2mm]
Z & 0 & 0 & 0 & 0\\[2mm]
W & k\alpha/k\alpha & 0 & 0 & 0
\end{array}
\]
and hence all entries vanish. Therefore,
$$
\Hom_{K^b(\proj\Lambda)}(M,\Sigma M)=0.
$$
Since $M$ is two-term, we also have
$$
\Hom_{K^b(\proj\Lambda)}(M,\Sigma^iM)=0
~~\text{for all}~i\ge2.
$$
Thus $M$ is presilting.

It remains to show that $M$ generates $K^b(\proj\Lambda)$. Since $W=\Sigma P_2$, we have
$P_2\in\thick(M)$, and $Z=P_3$ gives $P_3\in\thick(M)$. The complexes $X$ and $Y$ induce triangles
$$
P_2\xrightarrow{~\alpha~}P_1\longrightarrow X\longrightarrow\Sigma P_2
$$
and
$$
P_4\xrightarrow{~\gamma~}P_3\longrightarrow Y\longrightarrow\Sigma P_4.
$$
It follows that $P_1,P_4\in\thick(M)$. Hence all indecomposable projective $\Lambda$-modules belong to $\thick(M)$, and therefore
$$
\thick(M)=K^b(\proj\Lambda).
$$
Thus $M$ is silting.
\end{proof}

\begin{remark}\label{rem:M-not-tilting}
The silting object $M$ is not tilting. Indeed, since $W=\Sigma P_2$, the nonzero path
$$
\beta\colon P_3=Z\longrightarrow P_2=\Sigma^{-1}W
$$
gives a nonzero element of $\Hom_{K^b(\proj\Lambda)}(M,\Sigma^{-1}M)$. Consequently, the criterion in
\cite[Corollary~5.2]{XieYangZhang2023}, which ensures that the associated silting complex is tilting when the original two-term silting complex is tilting, cannot be applied to the present example.
\end{remark}

\hspace{-5mm}{\bf 3.3 A tilting object finitely presented by the silting object.}~
We construct a tilting object in $\pr(M)$ and describe its presentation by $M$.
Set
$$
T=\Sigma\Lambda.
$$
Since $\Lambda$ is the standard tilting object of $K^b(\proj\Lambda)$, the shifted object $T$ is tilting as well. We construct an explicit $M$-presentation of $T$.

Let
$$
E=X\oplus Z\oplus Z.
$$
Then
$$
E^{-1}=P_2,~
E^0=P_1\oplus P_3\oplus P_3,
~
d_E=
\begin{pmatrix}
\alpha\\
0\\
0
\end{pmatrix}.
$$
Regard $\Lambda=P_1\oplus P_2\oplus P_3\oplus P_4$ as a stalk complex concentrated in degree $0$. Define a chain map $e\colon\Lambda\to E$ by
$$
e^0=
\begin{pmatrix}
1_{P_1}&0&0&0\\
0&0&1_{P_3}&0\\
0&0&0&\gamma
\end{pmatrix}.
$$
We use the cochain mapping cone convention
$$
\Cone(e)^n=E^n\oplus\Lambda^{n+1},~
d_{\Cone(e)}^n=
\begin{pmatrix}
d_E^n&e^{n+1}\\
0&-d_\Lambda^{n+1}
\end{pmatrix}.
$$
Put $C=\Cone(e)$. Then
\begin{equation}\label{eq:cone-triangle}
\Lambda\xrightarrow{~e~}E\xrightarrow{~f~}C\longrightarrow\Sigma\Lambda
\end{equation}
is a triangle.

With the direct sum order
$$
C^{-1}=P_2\oplus P_1\oplus P_2\oplus P_3\oplus P_4,
~
C^0=P_1\oplus P_3\oplus P_3,
$$
the differential is
\begin{equation}\label{eq:dC}
d_C=
\begin{pmatrix}
\alpha&1&0&0&0\\
0&0&0&1&0\\
0&0&0&0&\gamma
\end{pmatrix}.
\end{equation}
The morphism $f$ is the canonical cone inclusion. Explicitly, $f^{-1}$ identifies $P_2=E^{-1}$ with the first summand of $C^{-1}$ and $f^0=1_{E^0}$.

\begin{lemma}\label{lem:C-addM}
In $K^b(\proj\Lambda)$, there exists an isomorphism
$$
C\cong W^{\oplus2}\oplus Y.
$$
In particular, $E,C\in\add(M)$.
\end{lemma}

\begin{proof}
Consider the first two summands in the source of \eqref{eq:dC}. The corresponding part of the differential is
$$
P_2\oplus P_1\xrightarrow{(\alpha,~1)}P_1.
$$
The automorphism
$$
\begin{pmatrix}
1&0\\
-\alpha&1
\end{pmatrix}
$$
of $P_2\oplus P_1$ transforms this map into $(0,1)$. Hence this summand decomposes as the direct sum of $W=(P_2\to0)$ and the contractible complex $(P_1\xrightarrow{1}P_1)$.
The third source summand in \eqref{eq:dC} gives another copy of $W$. The fourth source summand forms the contractible complex $(P_3\xrightarrow{1}P_3)$, while the last summand gives $Y=(P_4\xrightarrow{\gamma}P_3)$.
Removing the contractible summands gives the asserted isomorphism.
\end{proof}

Rotating \eqref{eq:cone-triangle}, we obtain
\begin{equation}\label{eq:M-presentation-T}
E\xrightarrow{~f~}C\longrightarrow T\longrightarrow\Sigma E.
\end{equation}
Since $E,C\in\add(M)$, the following consequence is immediate.

\begin{corollary}\label{cor:T-prM}
The object $T=\Sigma\Lambda$ is tilting and belongs to $\pr(M)$.
\end{corollary}

By Proposition~\ref{prop:M-silting} and Corollary~\ref{cor:T-prM}, all assumptions in \cite[Question~3.15]{Yang2025} are satisfied. Namely, $K^b(\proj\Lambda)$ is idempotent complete and algebraic, $M$ is silting in $K^b(\proj\Lambda)$, and $T$ is a tilting object in $\pr(M)$. Therefore, it remains to examine whether the associated two-term silting complex $\PP_M(T)$ admits a negative self extension.

\section{The negative self extension}
{\bf 4.1 The relevant morphism space.}~
We now compute the negative extension arising from the presentation constructed above.

\begin{lemma}\label{lem:HomCE}
We have
$$
\dim_k\Hom_{K^b(\proj\Lambda)}(C,E)=1.
$$
More precisely, this space is generated by the homotopy class of the chain map
$g\colon C\to E$ with $g^0=0$ and
\begin{equation}\label{eq:g}
g^{-1}=
\begin{pmatrix}
0&0&0&0&\beta\gamma
\end{pmatrix}
\colon C^{-1}\longrightarrow E^{-1}=P_2.
\end{equation}
\end{lemma}

\begin{proof}
Let $u\colon C\to E$ be a chain map. By \eqref{eq:homspaces-projectives}, the degree $-1$ component has the form
$$
u^{-1}=
\begin{pmatrix}
r&0&s&t\beta&v\beta\gamma
\end{pmatrix},
$$
and the degree $0$ component has the form
$$
u^0=
\begin{pmatrix}
A&B\alpha\beta&c\alpha\beta\\[1mm]
0&D&E_1\\[1mm]
0&F&G
\end{pmatrix}.
$$
The chain map equation
$$
d_Eu^{-1}=u^0d_C
$$
implies
$$
A=r=s=D=E_1=F=G=0,~ B=t,
$$
where $t,v,c\in k$ are arbitrary. Hence the space of chain maps from $C$ to $E$ has dimension three, with parameters $(t,v,c)$.

A chain homotopy is given by a morphism
$$
h\colon C^0=P_1\oplus P_3\oplus P_3\longrightarrow E^{-1}=P_2.
$$
By \eqref{eq:homspaces-projectives}, such a morphism has the form
$$
h=
\begin{pmatrix}
0&\lambda\beta&\mu\beta
\end{pmatrix}.
$$
The induced null-homotopic chain map $d_Eh+hd_C$ corresponds to the parameters
$$
(t,v,c)=(\lambda,\mu,\mu).
$$
Therefore the subspace of null-homotopic maps is
$$
\{(\lambda,\mu,\mu)\mid \lambda,\mu\in k\}
=
\operatorname{span}\{(1,0,0),(0,1,1)\},
$$
which has dimension two. It follows that
$$
\dim_k\Hom_{K^b(\proj\Lambda)}(C,E)=3-2=1.
$$
The chain map $g$ in \eqref{eq:g} corresponds to $(t,v,c)=(0,1,0)$, which is not null-homotopic. Hence its homotopy class is nonzero and generates $\Hom_{K^b(\proj\Lambda)}(C,E)$.
\end{proof}

\hspace{-5mm}{\bf 4.2 The two annihilation conditions.~}
We verify the two annihilation conditions appearing in Lemma~\ref{lem:obstruction}
for the morphism obtained above.

\begin{lemma}\label{lem:annihilation}
For the map $g$ in {\rm Lemma~\ref{lem:HomCE}}, one has
$$
gf=0~\text{and}~ fg=0
$$
in $K^b(\proj\Lambda)$. Consequently,
$$
\{u\in\Hom_{K^b(\proj\Lambda)}(C,E)\mid uf=0=fu\}=k[g].
$$
\end{lemma}

\begin{proof}
The map $f^{-1}\colon E^{-1}=P_2\to C^{-1}$ is the inclusion into the first $P_2$-summand of $C^{-1}$, while $g^{-1}$ is supported only on the last $P_4$-summand. Since $g^0=0$, the equality $gf=0$ holds already at the level of complexes.

It remains to show that $fg$ is null-homotopic. The degree $0$ component of $fg$ is zero, and its degree $-1$ component is the morphism $\beta\gamma$ from the last $P_4$-summand of $C^{-1}$ to the first $P_2$-summand. Define
$$
h\colon C^0\longrightarrow C^{-1}
$$
to be zero on the first two summands $P_1\oplus P_3$ and on the third $P_3$-summand set
$$
h|_{P_3}=
\begin{pmatrix}
\beta\\-\alpha\beta\\0\\0\\0
\end{pmatrix}.
$$
Then $d_Ch=0$, since the first row of $d_C$ gives
$
\alpha\beta-\alpha\beta=0.
$
On the other hand, $hd_C$ vanishes on every source summand except $P_4$, on which it is
$$
\begin{pmatrix}
\beta\gamma\\-\alpha\beta\gamma\\0\\0\\0
\end{pmatrix}
=
\begin{pmatrix}
\beta\gamma\\0\\0\\0\\0
\end{pmatrix},
$$
using $\alpha\beta\gamma=0$. Hence
$
fg=d_Ch+hd_C,
$
so $fg$ is null-homotopic.

Finally, Lemma~\ref{lem:HomCE} gives
$$
\dim_k\Hom_{K^b(\proj\Lambda)}(C,E)=1,
$$
and the class $[g]$ is nonzero. Therefore the subspace satisfying the two annihilation conditions is precisely $k[g]$.
\end{proof}

\hspace{-5mm}{\bf 4.3 Passage to the endomorphism algebra.}~
We now pass to the endomorphism algebra of the silting object. Let
$$
B=\End_{K^b(\proj\Lambda)}(M)
$$
and let
$$
F=\Hom_{K^b(\proj\Lambda)}(M,-)\colon\add(M)\xrightarrow{~\sim~}\proj B.
$$
The presentation \eqref{eq:M-presentation-T} induces
$
\PP_M(T)=\bigl(F(E)\xrightarrow{F(f)}F(C)\bigr).
$

\begin{proposition}\label{prop:negative-dim}
There exists an isomorphism of $k$-vector spaces
$$
\Hom_{K^b(\proj B)}
\bigl(\PP_M(T),\Sigma^{-1}\PP_M(T)\bigr)
\cong k.
$$
In particular,
$$
\dim_k\Hom_{K^b(\proj B)}
\bigl(\PP_M(T),\Sigma^{-1}\PP_M(T)\bigr)=1.
$$
\end{proposition}

\begin{proof}
By Lemma~\ref{lem:obstruction}, the left hand side is
$$
\{\varphi\in\Hom_B(F(C),F(E))\mid
\varphi F(f)=0=F(f)\varphi\}.
$$
The functor $F$ is fully faithful on $\add(M)$, and $E,C\in\add(M)$ by Lemma~\ref{lem:C-addM}.
It therefore identifies the preceding space with
$
\{u\in\Hom_{K^b(\proj\Lambda)}(C,E)\mid uf=0=fu\}.
$
Lemma~\ref{lem:annihilation} identifies the latter space with $k[g]$.
\end{proof}

\begin{proof}[\bf \emph{Proof of Theorem~\ref{thm:main}}]
Proposition~\ref{prop:M-silting} proves part~(1), and Corollary~\ref{cor:T-prM} proves part~(2).
Part~(3) follows from Proposition~\ref{prop:negative-dim}.
Indeed, a tilting complex has no nonzero self extensions in degree $-1$, whereas
$$
\Hom_{K^b(\proj B)}
\bigl(\PP_M(T),\Sigma^{-1}\PP_M(T)\bigr)\neq0.
$$
Hence $\PP_M(T)$ is not tilting.
On the other hand, $T$ is a two-term silting object in $\pr(M)$ in the terminology of \cite{Yang2025}. Therefore, by \cite[Theorem~3.14(c)]{Yang2025}, $\PP_M(T)$ is a two-term silting complex. This completes the proof.
\end{proof}

\section{Why the counterexample works}

The preceding computation admits a simple interpretation in terms of the path relations in $\Lambda$.
The nonzero morphism responsible for the negative extension is induced by the path
$$
P_4\xrightarrow{\beta\gamma}P_2\xrightarrow{\alpha}P_1.
$$
The three relations in \eqref{eq:path-relations} enter the construction in different ways.

First, the condition $\beta\gamma\neq0$ ensures that the corresponding chain map is nonzero.
Second, the relation
$
\alpha\beta\gamma=0
$
is precisely the chain-map condition
$$
d_Eg^{-1}=\alpha\beta\gamma=0.
$$
Third, $\alpha\beta\neq0$ prevents this chain map from being null-homotopic: any possible homotopy from a $P_3$-summand to $P_2$ is a scalar multiple of $\beta$, whose composition with $\alpha$ gives the nonzero morphism $\alpha\beta$.
Thus the vanishing of the length-three path, together with the nonvanishing of its length-two subpaths, produces the negative extension.

The projective dimension criterion of \cite{XieYangZhang2023} also clarifies the position of this example. Recall that
$$
E=X\oplus P_3\oplus P_3.
$$
The zeroth cohomology of $X$ is the simple module $S_1$, with minimal projective resolution
$$
0\longrightarrow P_4
\xrightarrow{\beta\gamma}P_2
\xrightarrow{\alpha}P_1
\longrightarrow S_1\longrightarrow0.
$$
Indeed, $\ker(\alpha)$ is generated by $\beta\gamma$. Hence
$$
H^0(E)\cong S_1\oplus P_3^{\oplus2},
$$
and
$$
\pd S_1=2,~ \pd H^0(E)=2.
$$
Xie, Yang, and Zhang \cite[Lemma~5.5]{XieYangZhang2023} proved that the associated two-term silting complex is tilting when the corresponding module $H^0(E)$ has projective dimension at most one.
Here the projective dimension is two, so their sufficient condition does not apply.
\vspace{8mm}

\hspace{-5.5mm}\textbf{Data Availability}\hspace{2mm} Data sharing not applicable to this article as no datasets were generated or analysed during
the current study.

\hspace{-5.5mm}\textbf{Conflict of Interests}\hspace{2mm} The authors declare that they have no conflicts of interest to this work.

%

\hspace{-5mm}\textbf{Jing He}\\
School of Mathematics and Statistics, Hunan University of Technology and Business, 410205 Changsha, Hunan P. R. China\\
E-mail: jinghe1003@163.com\\[0.4cm]
\textbf{Panyue Zhou}\\
School of Mathematics and Statistics, Changsha University of Science and Technology, 410114 Changsha, Hunan, P. R. China\\
E-mail: panyuezhou@163.com


\begin{thebibliography}{99}
\bibitem[AI]{AiharaIyama2012}
T. Aihara, O. Iyama,
\emph{Silting mutation in triangulated categories},
J. Lond. Math. Soc. (2) \textbf{85} (2012), no.~3, 633--668.

\bibitem[BZ]{BuanZhou2016}
A. B. Buan, Y. Zhou,
\emph{A silting theorem},
J. Pure Appl. Algebra \textbf{220} (2016), no.~7, 2748--2770.

\bibitem[XYZ]{XieYangZhang2023}
Z. Xie, D. Yang, H. Zhang,
\emph{A differential graded approach to the silting theorem},
J. Pure Appl. Algebra \textbf{227} (2023), Paper No.~107180, 23 pp.

\bibitem[Y]{Yang2025}
D. Yang,
\emph{From objects finitely presented by a rigid object in a triangulated category to 2-term complexes},
arXiv:2509.08246v1 (2025).



\end{thebibliography}
\end{document}